\documentclass[12pt,leqno]{amsart}
\usepackage{amscd}
\usepackage{color}
\usepackage[utf8]{inputenc}
\usepackage{amsxtra}
\usepackage{amsopn}
\usepackage{amsmath,amsthm,amssymb}
\usepackage{amscd}
\usepackage{amsfonts}
\usepackage{latexsym}
\usepackage{verbatim}
\usepackage{color}
\theoremstyle{plain}
\newtheorem{theorem}{Theorem}[section]
\newtheorem*{theorem*}{Theorem}

\newtheorem{lemma}[theorem]{Lemma}
\newtheorem{prop}[theorem]{Proposition}
\newtheorem{cor}[theorem]{Corollary}
\newtheorem{rem}[theorem]{Remark}
\newtheorem{example}[theorem]{Example}
\newtheorem*{mt*}{Main Theorem}


\newcommand{\del}{\partial}
\newcommand{\delbar}{\overline{\del}}

\newcommand\g{{\mathfrak g}}

\title[Pluriclosed and SKL metrics compatible with abelian complex structures]{Pluriclosed and Strominger K\"ahler-like metrics compatible with abelian complex structures} 

\author{Anna Fino}
\address[Anna Fino]{Dipartimento di Matematica ``G. Peano'' \\
Universit\`{a} degli studi di Torino \\
Via Carlo Alberto 10\\
10123 Torino, Italy
}
\email{annamaria.fino@unito.it}

\author{Nicoletta Tardini}
\address[Nicoletta Tardini]{Dipartimento di Scienze Matematiche, Fisiche e Informatiche\\
Unit\`a di Matematica e Informatica\\
Universit\`a degli Studi di Parma\\
Parco Area delle Scienze 53/A\\
43124 Parma, Italy}
\email{nicoletta.tardini@unipr.it}

\author{Luigi Vezzoni}
\address[Luigi Vezzoni]{Dipartimento di Matematica ``G. Peano'' \\
Universit\`{a} degli studi di Torino \\
Via Carlo Alberto 10\\
10123 Torino, Italy
}
\email{luigi.vezzoni@unito.it}


\begin{document}

\maketitle

\begin{abstract}
We show that the existence of a left-invariant pluriclosed Hermitian metric on a unimodular Lie group with a left-invariant abelian complex structure forces the group to be $2$-step nilpotent. Moreover, we prove that the pluriclosed flow starting from a left-invariant Hermitian metric on a $2$-step nilpotent Lie group preserves the Strominger K\"ahler-like condition.  
\end{abstract}

\section{Introduction}

A Hermitian metric $g$ on a complex manifold $(M,J)$ is called {\em pluriclosed} (or {\em SKT}) if its fundamental form 
$\omega(\cdot,\cdot)=g(J\cdot,\cdot)$ satisfies 
\begin{equation}\label{SKT}
dJ d \omega=0\,.
\end{equation}
The pluriclosed condition \eqref{SKT} can be characterized in terms of the torsion of the Bismut  (or Strominger) connection $\nabla^B$. Indeed, in  \cite{bismut} Bismut proved that  on a  Hermitian manifold $(M,J,g)$ there is a unique Hermitian connection $\nabla^B$  whose torsion $T^B$, once regarded as a $(3,0)$-tensor via $g$, is skew-symmetric.  The pluriclosed condition is equivalent to $dT^B=0$.  If $T^B =0$, the Bismut connection $\nabla^B$ coincides withe the Levi-Civita condition and the metric $g$ is K\"ahler.

By  \cite{zhao-zheng}  a  Hermitian  metric $g$  is  pluriclosed    and satisfies the condition    $\nabla^BT^B=0$  if and only if  its Bismut  curvature $R^B$  satisfies the first Bianchi identity 
\begin{equation}\label{first-bianchi}
\sigma_{x,y,z}R^B(x,y,z)=0
\end{equation}
and the type condition
\begin{equation}\label{J-invariance}
R^B(x,y,z)=R^B(Jx,Jy,z),
\end{equation}
for any tangent vectors $x,y,z $  in $M$.  Hermitian metrics satisfying  \eqref{first-bianchi} and \eqref{J-invariance}   are called in literature {\em Strominger K\"ahler-like} and  have been studied  recently in \cite{angella-otal-ugarte-villacampa, zhao-zheng, yau-zhao-zheng, fino-tardini}. 

  An important tool in the geometry of pluriclosed metrics is the so-called {\em pluriclosed flow} which is a parobolic flow of Hermitian metrics which preserves the pluriclosed condition \cite{streets-tian-1, streets-tian-2}. A natural question is to see if the Strominger  K\"ahler-like condition is preserved by the flow.

Every  conformal class of any Hermitian metric  on a compact complex surface  admits a pluriclosed  metric, but  in higher dimensions, the existence of a pluriclosed   metric is not automatically guaranteed anymore.   Looking at  the existence of  left-invariant   pluriclosed  metrics     on  $6$-dimensional nilpotent Lie groups  endowed with a left-invariant complex structure, only 4 out of the 34  isomorphism  classes  admit pluriclosed  metrics  and they are all  two-step nilpotent, leading to the question whether this is a general feature in arbitrary dimensions \cite{fino-vezzoni-2}.  It turns  out that two of the $4$ classes in dimension six  admit  Strominger K\"ahler-like metrics \cite{angella-otal-ugarte-villacampa}   and that the complex structure is  abelian. 
More in general,  a characterization of $2$-step  nilpotent Lie algebras  admitting Strominger K\"ahler-like metrics  have been  obtained in \cite{zhao-zheng-nilmanifolds}, showing in particular that the left-invariant complex structure  has to be abelian.

We recall that  a  left-invariant complex structure  on a real  Lie group $G$ of real dimension $2n$  is  completely determined by  a complex structure $J$   on the Lie algebra $\frak g$ of $G$, i.e. by  an endomorphism satisfying $J^2=-\text{Id}$ and  the integrability condition
$$
J[x,y]-[Jx,y]-[x,Jy]-J[Jx,Jy]=0,  \quad  \forall x,y\in\g\,.
$$
 The complex structure  
$J$ is called  \emph{abelian} if
\begin{equation}\label{Jabelian}
[Jx,Jy]=[x,y], \quad  \forall x,y\in\g\,.
\end{equation}
or  equivalently  if  the $i$-eigenspace of $J$, denoted with $\g^{1,0}$, is an abelian subalgebra of $\g^{\mathbb{C}}:=\g\otimes_{\mathbb{R}}\mathbb{C}$ (that motivates the terminology introduced in \cite{barberis-dotti-miatello}). By  \cite{petravchuk}   a Lie algebra admitting an abelian complex structure  has  abelian commutator, thus, it is 2-step solvable.

Recent results about the existence of pluriclosed  metrics on solvable Lie groups have been obtained in  \cite{madsen-swann, fino-otal-ugarte,arroyo-lafuente,fino-paradiso,freibert-swann}.

\smallskip

The purpose of this paper is twofold. On one hand we study the existence of a  pluriclosed  metric on a unimodular   Lie group with an abelian complex structure and on the other hand we investigate  the interplay between the  Strominger K\"ahler-like  condition  and the pluriclosed flow.    We recall that a Lie group $G$  is unimodular if and only if
$|det(Ad_g) |=1$, for every $g \in G$, where $Ad$ is the adjoint representation. For a connected Lie group $G$
 this is equivalent to requiring that $tr(ad_X)=0$, for every $X \in \frak g$, where $\frak g$ is the Lie algebra of $G$.

The existence of  other  types  of Hermitian inner products compatible with abelian complex structures, like for instance   \emph{K\"ahler} \cite{andrada-barberis-dotti-2},  \emph{balanced}   \cite{andrada-origlia}  and  \emph{locally conformally K\"ahler} inner products  \cite{andrada-origlia},  has been  already studied in letterature. In \cite{FKV} the  second  author and the third author, in collaboration with H.  Kasuya,  proved that on non-abelian Lie algebras with an abelian complex structure there are no Hermitian-symplectic structures. The latters can be regarded as special pluriclosed inner products and the natural follow-up is focusing on the existence of pluriclosed metrics compatible with abelian complex structures.

%
%
 \medskip

Our first result is the following 

\begin{theorem}\label{main}
Let $\mathfrak{g}$ be a unimodular Lie algebra with an abelian complex structure $J$. If $(\g,J)$ admits a {\rm pluriclosed} inner product, then $\g$ is $2$-step nilpotent.   
\end{theorem}

In the particular case when the commutator of $\g$ is totally real the result follows from \cite[Corollary 5.7]{freibert-swann}, but our proof does not make use of the argument in \cite{freibert-swann}.  Moreover, Theorem \ref{main} generalizes \cite[Proposition 6.1]{FKV}. 

\medskip 
 Next we focus on  the existence of   Strominger K\"ahler-like metrics in relation to the pluriclosed flow. By  using the characterization in \cite{zhao-zheng-nilmanifolds}   of  left-invariant Strominger K\"ahler-like metrics on $2$-step nilpotent Lie groups, we  prove the following

\begin{theorem}\label{main2} Let $(G,J,g_0)$ be a $2$-step nilpotent Lie group with a left-invariant  Strominger K\"ahler-like Hermitian structure
and let $g_t$ be the solution to the pluriclosed flow starting from $g_0$.  Then  $g_t$
is Strominger K\"ahler-like for every $t$.    
\end{theorem}

{\it Acknowledgements.}   The paper is supported by Project PRIN 2017 \lq \lq Real and complex manifolds: Topology, Geometry and Holomorphic Dynamics" and by GNSAGA of INdAM.

\section{Proof of Theorem \ref{main}}

We first need the following

\begin{lemma}
Let $\g$  be a   Lie algebra with an abelian complex structure $J$  and an Hermitian inner product $g$.
Then, the torsion $3$-form $T^B$  of the Bismut connection of  $(\g,J,g)$  satisfies 
$$
T^B(x,y,z)=-g([x,y],z)-g([y,z],x)-g([z,x],y),
$$
for every $x,y,z,w\in \mathfrak g$. 
\end{lemma}
\begin{proof}
Let $\omega$ be the fundamental form of $g$. 
Let $x,y,z,w\in \mathfrak g$, then $T^B(x,y,z)=-d\omega(Jx,Jy,Jz)$ and we directly compute
$$
\begin{aligned}
d\omega(Jx,Jy,Jz)&=-\omega([Jx,Jy],Jz)-\omega([Jy,Jz],Jx)-\omega([Jz,Jx],Jy)\\
&=-\omega([x,y],Jz)-\omega([y,z],Jx)-\omega([z,x],Jy)\,.
\end{aligned}
$$
Hence the claim follows.
\end{proof}

As a consequence we have

\begin{prop}\label{propSKT}
Let $(\g,J)$ be a Lie algebra with an abelian complex structure. 
A Hermitian inner product $g$ on  $(\g,J)$ is pluriclosed if and only if 
\begin{equation}\label{eqSKT}
g([y,z],[w,x])-g([x,z],[w,y])+g([x,y],[w,z])=0
\end{equation}
for every $x,y,z,w\in \mathfrak g$. 
\end{prop}
\begin{proof}
We recall that $g$ is pluriclosed if and only if $dT^B=0$.
Let $x,y,z,w\in \mathfrak g$, then, by the previous Lemma,
$$
\begin{aligned}
dT^B(w,x,y,z)&=-T^B([w,x],y,z)+T^B([w,y],x,z)-T^B([w,z],x,y)\\
&\quad-T^B([x,y],w,z)+T^B([x,z],w,y)-T^B([y,z],w,x)\\
&=g([[w,x],y],z)+g([y,z],[w,x])+g([z,[w,x]],y)\\
&\quad -g([[w,y],x],z)-g([x,z],[w,y])-g([z,[w,y]],x)\\
&\quad +g([[w,z],x],y)+g([x,y],[w,z])+g([y,[w,z]],x)\\
&\quad +g([[x,y],w],z)+g([w,z],[x,y])+g([z,[x,y]],w)\\
&\quad -g([[x,z],w],y)-g([w,y],[x,z])-g([y,[x,z]],w)\\
&\quad +g([[y,z],w],x)+g([w,x],[y,z])+g([x,[y,z]],w)\\
&= 2(g([y,z],[w,x])-g([x,z],[w,y])+g([x,y],[w,z])), \\
\end{aligned}
$$
where, in the last equality, we used the Jacobi identity.

Therefore, $g$ is pluriclosed if and only if
$$
g([y,z],[w,x])-g([x,z],[w,y])+g([x,y],[w,z])=0\,,
$$ 
as required. 
\end{proof}

\begin{rem}
{\rm  Note that  from the complex point of view,  condition \eqref{eqSKT}  is equivalent to
$$
g([z_1,\bar z_2],[z_3,\bar z_4])=g([z_1,\bar z_4],[z_3,\bar z_2]), 
$$ 
for every $z_1,z_2,z_3,z_4\in \g^{1,0}$.
}
\end{rem}

From now on, for a Lie algebra $\g$ with an abelian  complex structure $J$ we will  denote by $\zeta$ the center of $\g$ and by $\g^1_J$ the ideal
$$
\g^1_J=[\g,\g]+	J[\g,\g]\,.
$$  
Note that $\g^1_J$ is a $J$-invariant Lie subalgebra of $\g$.\\
Under the hypothesis of Proposition \ref{propSKT}  we  obtain the following  characterization     in terms of the center $\zeta$ of  $\g$.

\begin{cor}\label{corSKT}
Let $(\g,J,g)$ be a Lie algebra with an abelian complex structure and a pluriclosed inner product. Then 
$$
\| [x,y]\|^2+\| [x,Jy]\|^2=g([x,Jx],[y,Jy])
$$
for every $x,y\in \g$. In particular, $x\in\g$ lies in the center of $\g$ if and only if 
$$
[x,Jx]=0\,,
$$ 
i.e., 
$$
\zeta=\{x\in \g\,\,:\,\,[x,Jx]=0\}\,. 
$$
\end{cor}
\begin{proof}
By using \eqref{eqSKT} for $x,y\in \g$ we have
$$
\begin{aligned}
\| [x,y]\|^2+\| [x,Jy]\|^2=&\, g([x,y],[x,y])+g([x,Jy],[x,Jy])\\
=&\,g([Jx,Jy],[x,y])-g([Jx,y],[x,Jy])=g([y,Jy],[x,Jx])\,,
\end{aligned}
$$
and the claim follows. 
\end{proof}

We will need the following

\begin{lemma}\label{lemmagin}
Let $\mathfrak{g}$ be a unimodular Lie algebra with an abelian complex structure $J$. Then,
$$
\mathfrak{g}^1_J\neq\mathfrak{g}\,.
$$
\end{lemma}
\begin{proof}
By contradiction, assume that $\mathfrak{g}^1_J=\mathfrak{g}$. Then, since by hypothesis $\mathfrak{g}^1$ is an abelian ideal in $\mathfrak{g}$, by \cite[Proposition 4.1]{barberis-dotti} $(\mathfrak{g}/\zeta,J)$ is holomorphically isomorphic to $\mathfrak{aff}(\mathcal{A})$ for some commutative algebra $\mathcal{A}$. Since, $\mathfrak{g}$ is unimodular, also $\mathfrak{g}/\zeta$ is unimodular, and so $\mathfrak{aff}(\mathcal{A})$ is unimodular. So, by \cite[Lemma 2.6]{andrada-origlia}, $\mathcal{A}$ is nilpotent and $\mathfrak{aff}(\mathcal{A})$ is a nilpotent Lie algebra. As a consequence, we have that $\mathfrak{g}/\zeta$ is also nilpotent implying that $\mathfrak{g}$ is nilpotent too. But, this is absurd since by \cite{salamon} for a nilpotent Lie algebra $\mathfrak{g}$ we have $\mathfrak{g}^1_J\neq\mathfrak{g}$.
\end{proof}

%
%


\begin{prop}\label{gin}
Let $\mathfrak{g}$ be a Lie algebra with an abelian complex structure $J$. Assume that $(\g,J)$ has a {\rm pluriclosed} inner product $g$ and $\g^1_{J}$ is $2$-step nilpotent. Then $\g$ is $2$-step nilpotent. 
\end{prop}

\begin{proof}
Write 
$$
\g=(\g_{J}^1)^{\perp}\oplus  \g_{J}^1
$$
with respect to  the inner product $g$. Since $\g_{J}^1$ is nilpotent and has a pluriclosed inner product, its center $\mathfrak u$ is $J$-invariant. We write 
$$
\g_{J}^1=\mathfrak u^{\perp}\oplus \mathfrak u
$$  
The key observation is that $\mathfrak u$ is contained in the center of $\g$.
 Indeed, if $x\in \mathfrak u$, then in particular we have 
$[x,Jx]=0$ and Corollary \ref{corSKT} implies that $x$ belongs to the center of $\g$.\\
Now let $f\in (\g_{J}^1)^{\perp}$. We show that $[f,x]$ lies in the center of $\g$, for every $x\in\g$.\\
Set $D:={\rm ad}_f\colon \g_{J}^1\to \g_{J}^1$. Since $J$ is abelian, $ad_f J = - ad_{Jf}$ and therefore, 
 $D$ and $DJ$ are both derivations. Moreover, $D[x,y]=0$ for every $x,y\in \g^1_{J}$; indeed from the $2$-step nilpotency of $\g^1_{J}$, we have that $[x,y]\in\mathfrak u$, for every $x,y\in \g^1_{J}$, and that $\mathfrak u\subset\zeta$.

 Let $x\in \g^1_J$ and $y\in \g$. Then, taking into account that $\g$ is $2$-step solvable, Corollary \ref{corSKT} yields that  
$$
\|[Dx,y]\|^2+\|[D x,Jy]\|^2=g([Dx,JDx],[y,Jy])=g([DJx,Dx],[y,Jy])=0\,,
$$
from which we deduce that $[f,x]$ is in the center of $\g$ for all $x\in \g^1_J$.

Now let $f_1,f_2\in (\g_J^1)^\perp$. By Jacobi identity 
$$
[[f_1,f_2],x]=0
$$
for every $x\in \g^1_J$. Hence $[f_1,f_2]\in \mathfrak u$ and so in the center of $\g$, as required.   
\end{proof}

Now we are ready to prove Theorem \ref{main}.

\begin{proof}[Proof of Theorem $\ref{main}$]
We work by induction on the complex dimension $n$ of $\g$. The base case $n=1$ is trivial and we assume that the statement holds up to complex dimension $n-1$. Let $(\g,J,g)$ be a Lie algebra of complex dimension $n$ with an abelian complex structure and a pluriclosed inner product. In view of Lemma \ref{lemmagin}, $\g_J^1$ is a proper Lie 
subalgebra and inherits an abelian complex structure and a pluriclosed inner product.  By induction assumption $\g_J^1$ is $2$-step nilpotent.  Hence Proposition \ref{gin} implies that $\g$ is  $2$-step nilpotent and the claim follows.

%
%
%

\end{proof}

\begin{rem}{\rm
By Theorem \ref{main}, if $\mathfrak{g}$ is a unimodular Lie algebra with an abelian complex structure $J$ and a {\rm pluriclosed} inner product, then $\g$ is $2$-step nilpotent. In particular, notice that $\g^1_J$ is abelian. Indeed,
since $J$ is abelian, $\g$ is $2$-step solvable and 
$$
[\g^1,\g^1]=[J\g^1,J\g^1]=0\,. 
$$ 
Moreover from the $2$-step nilpotency of $\g$ we infer that also
$$
[\g^1,J\g^1]=0\,.
$$
As a consequence, if $X=\Gamma\backslash G$ is a nilmanifold endowed with an invariant  abelian complex structure $J$ and a pluriclosed metric $g$, then by \cite[Theorem A]{fino-vezzoni-2}, $X$ is a total space of a principal holomorphic torus bundle over a torus.}
\end{rem}

Notice that from Theorem \ref{main} in particular  follows that a nilpotent Lie algebra with an  abelian complex structure and 
admitting a pluriclosed inner product is necessarily $2$-step. This partially confirm the conjecture that the existence of a pluriclosed inner product on a nilpotent Lie algebra $\g$ with a complex structure forces $\g$ to be $2$-step.

\medskip 
Moreover, it is quite natural wondering how rigid is the existence of other kind of special inner products on a Lie algebra with a complex structure. In particular, the so-called \emph{astheno-K\"ahler} 
metrics introduced by Jost and Yau in \cite{jost-yau}, which are characterized by the condition
$$
 \del\delbar\omega^{ n - 2 } = 0\,.
$$
Clearly, on a complex surface any Hermitian metric is astheno-K\"ahler and in complex dimension $3$ the notion of astheno-K\"ahler metric coincides with that of pluriclosed.
 Here we observe that in the nilpotent case the existence of a astheno-K\"ahler inner product on a Lie algebra compatible with an abelian complex structure does not force the $2$-step condition in contrast to Theorem \ref{main} for the pluriclosed case. 
\begin{example}
{\em In view of  \cite[Corollary 5.1.9]{latorre-thesis} we consider the $8$-dimensional $3$-step nilpotent Lie algebra $\g$ with complex structure equations
$$
d\varphi^1=d\varphi^2=0\,,\qquad d\varphi^3=\varphi^{1\bar 1}\,,\qquad
d\varphi^4=B_{1\bar1}\varphi^{1\bar 1}+
B_{1\bar3}(\varphi^{1\bar 2}+
\varphi^{1\bar 3})+
D_{3\bar1}(\varphi^{2\bar 1}+
\varphi^{3\bar 1})\,,
$$
with $D_{3\bar 1}\neq 0$.
In particular, the complex structure $J$ is abelian.\\
Let
$$
\omega =\sum_{k=1}^3ix_{k\bar k}\varphi^{k\bar k}+\sum_{1\leq k<l\leq 3}
(x_{k\bar l}\varphi^{k\bar l}-\bar x_{k\bar l}\varphi^{l\bar k}) +\frac{i}{2}\varphi^{4\bar4}.
$$
If $ix_{2\bar2}  +ix_{3\bar3}+2\Im m\,(x_{2\bar3})=0$,   then $\omega$ defines  an astheno K\"ahler metric on $(\g,J)$.
}
\end{example}

\section{Proof of Theorem \ref{main2}}

\medskip 
Let $G$ be a 2-step nilpotent  Lie group with a left-invariant Hermitian structure $(g,J)$ and denote by $\g$  its  Lie algebra. Assume further that $g$ is pluriclosed. 
In view of \cite{zhao-zheng-nilmanifolds}, the metric $g$ is Strominger K\"ahler-like if and only if  
there exists an orthonormal basis $\{x_i\}_{i=1}^{s}$ of $\mathfrak{g}^1=[\mathfrak g,\mathfrak g]$ and an orthonormal basis $\{\epsilon_i\}_{i=1}^{2n}$ of $\mathfrak g$ such that 
\begin{enumerate}

\vspace{0.1cm}
\item[1.] $J\epsilon_i=\epsilon_{i+n}\,,\quad i=1,\dots, n$; 

\vspace{0.1cm}
\item[2.] $\mathfrak g^1+J\mathfrak g^1={\rm span}\{\epsilon_{r+1},\dots, \epsilon_n,\epsilon_{n+r+1},\dots,\epsilon_{2n}\}$;

\vspace{0.1cm}
\item[3.] the  only non-trivial brackets under $\{\epsilon_i\}$ are
$$
[\epsilon_i,\epsilon_{n+i}]=\lambda_ix_i	\,,\quad i=1,\dots,s, 
$$
for some positive numbers $\{\lambda_i\}_{i=1}^s$ and
$n-r \leq s\leq \text{min}\left\lbrace r,2(n-r)\right\rbrace$.  
\end{enumerate}
Note, that in particular $J$ has to be abelian.

If $\{\epsilon^i\}$ is the dual basis to $\{\epsilon_i\}$, then 
the metric $g$ writes as 
$$
g=\sum_{k=1}^{n}  (\epsilon^{k})^2.
$$
From \cite{zhao-zheng-nilmanifolds} it follows that every other left-invariant pluriclosed metric $h$ taking the diagonal form 
$$
h=\sum_{k=1}^n a_k \epsilon^{k}\epsilon^{k}\,,\quad a_k>0, \mbox{ for every }k=1,\dots,n\,,
$$ 
is Strominger K\"ahler-like since we can modify the basis $\{\epsilon_k\}$ to 
$$
\tilde \epsilon_{k}=\frac{1}{\sqrt{a_k}}\epsilon_k
$$
which still satisfies items  1.,2.,3..

Moreover, in view of \cite{enrietti-fino-vezzoni},    
the Ricci form of the Bismut connection of $h$ takes the following expression 
$$
\rho^B_h(x,y) =\frac{1}{2}\sum_{k=1}^s\frac{1}{a_k}h([\epsilon_k,\epsilon_{k+n}],[x,y])\,.
$$
which implies that $\rho^B_h$ takes the  diagonal form 
$$
\rho^B_h=\sum_{k=1}^sb_k\epsilon^k\wedge \epsilon^{n+k}\,.
$$

It follows that, by uniqueness, the solution to the pluriclosed flow starting from $g_0$ is diagonal for every $t$ and the claim of Theorem \ref{main2} 
follows. 

\medskip 
\begin{rem}
{\rm
Notice that we can give a more explicit expression for the Ricci form  $\rho^B_{g_t}$ of the Bismut connection  $\nabla^B$ of the metric $g_t$. 
Let $\left\lbrace\epsilon_i\right\rbrace$ be a basis satisfying items 1.,2.,3.  and
$$
g_0=\sum  (\epsilon^{k})^2 \,.
$$
Consider the solution to the pluriclosed flow
$$
g_t=\sum a_k^t  (\epsilon^{k})^2\,.
$$
If  $\left\lbrace\epsilon_k^t\right\rbrace$  is  a $g_t$-orthonormal basis satisfying items 1.,2.,3., namely
$$
\epsilon_{k}^t=\frac{1}{\sqrt{a_k^t}}\epsilon_k\,,
$$
then
$$
\rho^B_{g_t}(x,y)=
\frac{1}{2}\sum g_t([\epsilon_{k}^t,\epsilon_{n+k}^t],[x,y])\,.
$$
We have
$$
[\epsilon_{k}^t,\epsilon_{n+k}^t]=\frac{1}{\sqrt{a_k^t}}\frac{1}{\sqrt{a_{n+k}^t}}[\epsilon_{k},\epsilon_{n+k}]=
\frac{1}{a_k^t}\lambda_k x_k
$$
and
$$
[\epsilon_{k}^t,\epsilon_{n+k}^t]=\lambda_k^t x_k^t
$$
with $\left\lbrace x_k^t\right\rbrace$  $g_t$-orthonormal.\\
Hence
$$
\rho^B_{g_t}(x,y)=
\frac{1}{2}\sum \lambda_k^t g_t(x_k^t,[x,y])\,.
$$
Now
$$
\rho^B_{g_t}(\epsilon_i,\epsilon_{n+i})=
\frac{1}{2}\sum_{k=1}^s \lambda_k^t g_t(x_k^t,[\epsilon_i,\epsilon_{n+i}])\,.
$$
Since
$$
[\epsilon_i,\epsilon_{n+i}]=a_i^t [\epsilon_{i}^t,\epsilon_{n+i}^t]=
a_i\lambda_i^t x_i^t
$$
we get
$$
\rho^B_{g_t}(\epsilon_i,\epsilon_{n+i})=
\frac{1}{2}\sum_{k=1}^s \lambda_k^t g_t(x_k^t,[\epsilon_i,\epsilon_{n+i}])=
\frac{1}{2}\sum_{k=1}^s \lambda_k^t a_i^t g_t(x_k^t,\lambda_i^t x_i^t)=
\frac{1}{2}(\lambda_i^t)^2a_i^t.
$$
Therefore 
$$
\rho^B_{g_t}=\frac{1}{2}\sum_{k=1}^s (\lambda_k^t)^2a_k^t
\epsilon^k\wedge\epsilon^{n+k}.
$$
}
\end{rem}


\begin{thebibliography}{12}




\bibitem {AndradaBarberisDotti}
A. Andrada, M. L. Barberis, I. Dotti, Classification of abelian complex structures on 6-dimensional Lie algebras, \emph{J. Lond. Math. Soc.} (2) \textbf{83} (2011), no. 1, 232--255.

\bibitem{AndradaBarberisDottiCorrigendum}
A. Andrada, M. L. Barberis, I. Dotti, Corrigendum: Classification of abelian complex structures on six-dimensional Lie algebras, \emph{J. Lond. Math. Soc. (2)} \textbf{87} (2013), no. 1, 319--320.

\bibitem{andrada-barberis-dotti-2}
A. Andrada, M. L. Barberis, I. Dotti, Abelian Hermitian geometry,  \emph{Differential Geom. Appl.} \textbf{30} (2012), no. 5, 509--519.

\bibitem{andrada-origlia} A. Andrada, M. Origlia, Locally conformally K\"ahler structures on unimodular Lie groups, \emph{Geom.  Dedicata}, \textbf{179(1)} (2015), 197--216

\bibitem{angella-otal-ugarte-villacampa} D. Angella, A. Otal, L. Ugarte, R. Villacampa, On Gauduchon connections with K\"ahler-like curvature, \texttt{arXiv:1809.02632 [math.DG]},  to appear in \emph{Commun. Anal. Geom.}

\bibitem{andrada-villacampa} A. Andrada, R. Villacampa, Abelian balanced Hermitian structures on unimodular Lie algebras. \emph{Transform. Groups} \textbf{21} (2016), no. 4, 903--927.

\bibitem{arroyo-lafuente} R. Arroryo, R. Lafuente, The long-time behavior of the homogeneous pluriclosed flow, \emph{Proc. Lond. Math. Soc.}  (3) \textbf{119} (2019), no. 1, 266--289.

\bibitem{barberis-dotti-miatello} M. L. Barberis, I. G. Dotti Miatello, R. J. Miatello, On certain locally homogeneous Clifford manifolds, \emph{Ann. Glob. Anal. Geom.}  \textbf{13} (1995), 289--301.

\bibitem{barberis-dotti} M.L. Barberis, I. Dotti, Abelian Complex Structures on Solvable Lie Algebras, \emph{J. Lie Theory} \textbf{14} (2004) 25--34.

\bibitem{bismut} J.-M. Bismut,  A local index theorem for non-K\"ahler manifolds,  \emph{Math. Ann.} \textbf{284} (1989), no. 4, 681--699.

\bibitem{dotti-fino} I. Dotti,  A. Fino, HyperK\"ahler torsion structures invariant by nilpotent Lie groups,  \emph{Classical Quantum Gravity} \textbf{19} (2002), no. 3, 551--562.

\bibitem{enrietti-fino-vezzoni} N. Enrietti, A. Fino, L. Vezzoni, The pluriclosed flow on nilmanifolds and Tamed symplectic forms, \emph{J. Geom. Anal.} \textbf{25} (2015), 883--909.

\bibitem{FKV}
A. Fino, H. Kasuya, L. Vezzoni,
SKT and tamed symplectic structures on solvmanifolds, {\em Tohoku Math. J. (2)} {\bf 67} (2015), 19--37.



\bibitem{fino-otal-ugarte} A. Fino,  A. Otal, L.Ugarte,  Six-dimensional solvmanifolds with holomorphically trivial canonical bundle, \emph{ Int. Math. Res. Not. IMRN}  {\bf 2015}, no. 24, 13757--13799. 

\bibitem{fino-paradiso} A. Fino, F. Paradiso,  Generalized K\"ahler almost abelian Lie groups, \texttt{arXiv:2008.00458}, to appear
in \emph{Ann. Mat. Pura Appl.}

\bibitem{FinoPartonSalamon}
A. Fino, M. Parton, S. Salamon, Families of strong KT structures in six dimensions,  \emph{Comment. Math. Helv.} \textbf{79} (2004), no. 2, 317--340.

\bibitem{fino-tardini} A. Fino, N. Tardini, 
Some remarks  on  Hermitian manifolds satisfying K\"ahler-like conditions, 
\texttt{arXiv:2003.06582 [math.DG]}, 2020,
to appear in \emph{Math. Z.}

\bibitem{fino-vezzoni-2}
A. Fino, L. Vezzoni, A correction to  \lq \lq Tamed symplectic forms and strong K\"ahler with torsion metrics'',  \emph{J. Symplectic Geom.} \textbf{17} (2019), no. 4, 1079--1081.

\bibitem{freibert-swann} 
M. Freibert, A. Swann, Two-step solvable SKT shears, {\tt arXiv:2011.04331}.


\bibitem{jost-yau} J. Jost, S.-T. Yau, A nonlinear elliptic system for maps from Hermitian to Riemannian manifolds and rigidity theorems in Hermitian geometry, \emph{Acta Math.} \textbf{170} (1993), no. 2, 221--254; Correction, \emph{Acta Math.}  \textbf{173}  (1994), no. 2, 307.

\bibitem{latorre-thesis} A. Latorre, Geometry of nilmanifolds with invariant complex structure, Phd Thesis, 2016

\bibitem{madsen-swann} T. Madsen, A. Swann,  Invariant strong KT geometry on four-dimensional solvable Lie groups, \emph{ J. Lie Theory}  \textbf{21}  (2011), no. 1, 55--70.

\bibitem{petravchuk} A. P. Petravchuk, Lie algebras decomposable into a sum of an abelian and a nilpotent subalgebra, \emph{Ukr. Math. J.} \textbf{40} (3)
(1988), 331--334.

\bibitem{pujia-vezzoni} M. Pujia, L. Vezzoni, A remark on the Bismut-Ricci form on 2-step nilmanifolds,  \emph{C. R. Math. Acad. Sci. Paris} \textbf{356} (2018), no. 2, 222--226.

\bibitem{salamon} S. Salamon, Complex structures on nilpotent Lie algebras, \emph{J. Pure Appl. Algebra} \textbf{157} (2001) 311--333.

\bibitem{streets-tian-1}  J. Streets, G. Tian, A Parabolic flow of pluriclosed metrics,  {\emph  {Int. Math. Res. Notices}}  \textbf {2010}
(2010), no. 16, 3101--3133.

\bibitem{streets-tian-2}  J. Streets, G. Tian, Regularity results for pluriclosed flow,  {\emph {Geom. Topol.}}  \textbf{17} (2013), no. 4,
2389--2429.

\bibitem{snow} J. E. Snow, Invariant complex structures on four dimensional solvable real Lie groups, \emph{{Manuscripta Math.}} \textbf{66} (1990) 397--412.

\bibitem{yau-zhao-zheng} S. T. Yau, Q. Zhao, F. Zheng,
On Strominger K\"ahler-like manifolds with degenerate torsion,
 \texttt{arXiv:1908.05322}.

 \bibitem{zhao-zheng} Q. Zhao, F. Zheng, Strominger connection and pluriclosed metrics, \texttt{arXiv:1904.06604}.

 \bibitem{zhao-zheng-nilmanifolds} Q. Zhao, F. Zheng, Complex nilmanifolds and K\"ahler-like connections, \emph{J. Geom. Phys.} \textbf{146} 2019, 103512.

\end{thebibliography}
\end{document}